\documentclass[a4paper]{article}
\usepackage[utf8]{inputenc}
\usepackage{amsmath,amssymb,amsthm,bbold,colortbl,enumerate,float,mathrsfs,mathtools,mdframed,physics,todonotes,url,xcolor}
\usepackage[top=2.5cm, bottom=2.5cm, left=2.4cm, right=2.4cm]{geometry}
\usepackage{algorithm,algpseudocode}
\usepackage[hidelinks]{hyperref}

\usepackage[noabbrev]{cleveref}
\newtheorem{theorem}{Theorem}
\newtheorem{lemma}[theorem]{Lemma}

\newtheorem{corollary}[theorem]{Corollary}

\theoremstyle{definition}

\newtheorem{example}[theorem]{Example}
\newtheorem{problem}[theorem]{Problem}
\newtheorem{definition}[theorem]{Definition}

\newcommand{\FF}{\mathbb F}
\newcommand{\CC}{\mathbb C}
\newcommand{\ZZ}{\mathbb Z}

\newcommand{\QLS}[1]{\(\mathrm{QLS}(#1)\)}
\newcommand{\MOLS}[1]{\(\mathrm{MOLS}(#1)\)}
\newcommand{\MOQLS}[1]{\(\mathrm{MOQLS}(#1)\)}

\title{Large sets of mutually orthogonal quantum Latin squares}
\author{Simeon Ball \\ \textit{Universitat Politècnica de Catalunya} \and Robin Simoens \\ \textit{Ghent University} \\ \textit{Universitat Politècnica de Catalunya}}
\date{}

\begin{document}

\maketitle

\begin{abstract}
    How large can a set of mutually orthogonal quantum Latin squares (MOQLS) get? We show that a set of \(n-2\) MOQLS of order \(n\) is necessarily classical and construct large non-classical sets of MOQLS of orders that are prime powers, improving both the previously known lower and upper bounds.
\end{abstract}



\section{Introduction}

Mutually orthogonal quantum Latin squares\footnote{See Section~\ref{sec:prelim} for their definition.} (MOQLS) were introduced by Musto \cite{Mustothesis} as a generalisation of mutually orthogonal Latin squares (MOLS). Quantum Latin squares, and more generally, MOQLS, have applications in quantum teleportation, dense coding, complex Hadamard matrices, quantum graph isomorphism, mutually unbiased bases, AME states and perfect tensors \cite{Qgraphisomorphism,nonclassical,motivation,2MOQLS7,Mustopaper,Mustothesis,GMOQLS}. 
A natural problem regarding these objects is their existence.

\begin{problem}
    Given \(n\geq2\), what is the largest value of \(t\) such that \(t\) \MOQLS{n} exist?
\end{problem}

Since we are mostly interested in the non-classical case, we define \(M(n)\) as the largest value of \(t\) such that a \emph{non-classical} set of \(t\) \MOQLS{n} exist, following the notation in \cite{GMOQLS}. If such a set does not exist, we let \(M(n)=0\).

The following results on the existence of non-classical sets of MOQLS are known.

\begin{theorem}[\cite{36paper}]
    A set of \(n-1\) \MOQLS{n} is classical. That is, \(M(n)\leq n-2\).
\end{theorem}

\begin{theorem}[\cite{GMOQLS}]
    If there exist \(t_1\) \MOQLS{n_1} and \(t_2\) \MOQLS{n_2} and if \(t=\min\{t_1,t_2\}\geq2\), then there exists a non-classical set of \(t\) \MOQLS{n_1n_2}.
\end{theorem}

\begin{theorem}[\cite{36paper,nonclassical,2MOQLS7,GMOQLS}]
    \begin{enumerate}
        \item There exists a non-classical \QLS{n} if and only if \(n\geq4\).
        \item There exists a non-classical set of \(2\) \MOQLS{n} if and only if \(n\geq7\).
        \item If \(n\geq16\), then there exists a non-classical set of \(3\) \MOQLS{n}.
    \end{enumerate}
\end{theorem}

\begin{table}[H]
    \centering
    \begin{tabular}{c|cccccccccccccccc}
\(n\) & 2 & 3 & 4 & 5 & 6 & 7 & 8 & 9 & 10 & 11 & 12 & 13 & 14 & 15 & 16 & 17\\
\hline
\(M(n)\) & 0 & 0 & 1 & 1 & 1 & 2-5 & 2-6 & 2-7 & 2-8 & 2-9 & 2-10 & 2-11 & 2-12 & 2-13 & 3-14 & 3-15
\end{tabular}
\caption{Previously known bounds on \(M(n)\) for \(2\leq n\leq17\).}
    \label{tab:oldbounds}
\end{table}

In this paper, we improve both lower and upper bounds on \(M(n)\). In Section~\ref{sec:n-2}, we prove the following.


\begin{theorem}\label{thm:n-2}
    Suppose \(n\geq3\). A set of \(n-2\) \MOQLS{n} is classical. Thus, \(M(n)\leq n-3\).
\end{theorem}


In Section~\ref{sec:construction}, we provide a construction of large non-classical sets of MOQLS, inspired by a construction of Huang and Li \cite{2MOQLS7} for non-classical sets of \(2\) \MOQLS{n} with $n\geq7$, \(n\) odd. This construction seems to naturally extend to Frobenius rings. For the definition of Frobenius rings, we refer to Section~\ref{sec:prelim}.


\begin{theorem}\label{thm:construction}
    Let \(R\) be a finite Frobenius ring of order \(n\) with unit group \(R^*\). Let \(f:R\to R\) be a map that is not the sum of an additive map and a constant. If \(M\subseteq R\) such that \[\forall m\in M\colon\,x\mapsto f(x)-mx \text{ is a permutation of }R\] and \[\forall m,m'\in M\colon\,m\neq m'\implies m-m'\in R^*,\] then there exist \(|M|\) \MOQLS{n}, exactly one of which is non-classical.
\end{theorem}

Examples of finite Frobenius rings are modular rings and finite fields.

Note that Theorem~\ref{thm:construction} is closely related to the concept of orthomorphisms and complete mappings: \(f\) is an \emph{orthomorphism} if both \(f\) and the map \(x\mapsto f(x)-x\) are permutations; it is called a \emph{complete mapping} if both \(f\) and the map \(x\mapsto f(x)+x\) are permutations. The construction of \(2\) \MOQLS{n}, \(n\geq7\) and \(n\) odd, given in \cite{2MOQLS7} is equivalent to a construction of a complete mapping \(f:\ZZ/n\ZZ\to\ZZ/n\ZZ\) that is not of the form \(x\mapsto ax+b\).

If \(n=q\) is a prime power, we can use the structure of \(\FF_q\), the finite field with \(q\) elements, to get good constructions of such functions determining few directions \cite{BBS,LS1981}. Using a family of such functions, we construct large sets of \(t\) \MOQLS{q} where \(q\) is a prime power.

\begin{corollary}\label{cor:construction}
    If \(q\) is a prime power and \(q-1\) has a proper divisor \(d\neq1\), then there exist \(d-1\) \MOQLS{q}, exactly one of which is non-classical.
\end{corollary}

Theorem~\ref{thm:n-2} and Corollary~\ref{cor:construction} yield improvements to the previous bounds on \(M(n)\) for small values of \(n\) in Table~\ref{tab:newbounds} (compare with Table~\ref{tab:oldbounds}).

\begin{table}[H]
    \centering
    \begin{tabular}{c|cccccccccccccccc}
\(n\) & 2 & 3 & 4 & 5 & 6 & 7 & 8 & 9 & 10 & 11 & 12 & 13 & 14 & 15 & 16 & 17\\
\hline
\(M(n)\) & 0 & 0 & 1 & 1 & 1 & 2-4 & 2-5 & 3-6 & 2-7 & 4-8 & 2-9 & 5-10 & 2-11 & 2-12 & 4-13 & 7-14
\end{tabular}
\caption{New bounds on \(M(n)\) for \(2\leq n\leq17\).}
    \label{tab:newbounds}
\end{table}

\section{Preliminaries}\label{sec:prelim}

\subsection{Latin squares}

A \emph{Latin square} of order \(n\) is an \(n\times n\) matrix containing elements of a set of size \(n\) such that each row and each column contains every element exactly once.
Two Latin squares \(A=\left(a_{ij}\right)_{1\leq i,j\leq n}\) and \(B=\left(b_{ij}\right)_{1\leq i,j\leq n}\) are \emph{orthogonal} if all \(n^2\) tuples \((a_{ij},b_{ij})\) are different.
A set of \emph{\(t\) mutually orthogonal Latin squares of order \(n\)}, abbreviated \(t\) MOLS\((n)\), is a set of \(t\) Latin squares of order \(n\) that are pairwise orthogonal.

\subsection{Quantum Latin squares}

A \emph{quantum Latin square of order \(n\)} is an \(n\times n\) matrix over \(\CC^n\) such that each row and each column forms an orthonormal basis of \(\CC^n\).
Two quantum Latin squares are \emph{isotopic} if they can be obtained from one another by
\begin{enumerate}[(i)]
    \item multiplying each entry by a phase factor,
    \item permuting rows and columns, and
    \item\label{item:u} applying a unitary transformation on all entries.
\end{enumerate}
A quantum Latin square is \emph{classical} or \emph{not genuinely quantum} if, up to isotopy, all entries are contained in \(\left\{\ket{1},\dots,\ket{n}\right\}\), a fixed orthonormal basis.
This property is independent of the chosen basis by item~(\ref{item:u}) in the definition of isotopy.

\subsection{Mutually orthogonal quantum Latin squares (MOQLS)}

Two quantum Latin squares \(\left(\psi_{ij}\right)_{1\leq i,j\leq n}\) and \(\left(\phi_{ij}\right)_{1\leq i,j\leq n}\) are \emph{orthogonal} if \[\left\{\psi_{ij}\otimes\phi_{ij}\colon\,i,j\in\{1,\dots,n\}\right\}\] is an orthonormal basis of \(\CC^n\otimes\CC^n\).
A set of \emph{\(t\) mutually orthogonal quantum Latin squares of order \(n\)}, denoted \(t\) \MOQLS{n}, is a set of \(t\) quantum Latin squares of order \(n\) such that any two of them are orthogonal.
Two sets of \(t\) MOQLS\((n)\) are \emph{isotopic} if they can be obtained from one another by
\begin{enumerate}[(i)]
    \item multiplying each entry by a phase factor,
    \item permuting rows and columns, simultaneously among all squares, and
    \item\label{item:lu} applying, separately to each square, a unitary transformation to all of its entries.
\end{enumerate}
We say that a set of \(t\) MOQLS\((n)\) is \emph{classical} or \emph{not genuinely quantum} if, up to isotopy, all entries are contained in \(\left\{\ket{1},\dots,\ket{n}\right\}\). Again, by item~(\ref{item:lu}) above, this property does not depend on the chosen orthonormal basis.

\subsection{Patterns of quantum Latin squares}

The \emph{pattern} of a matrix is the binary matrix that is obtained from it by replacing each nonzero entry by a one.
A \emph{unitary pattern} is a pattern of a unitary matrix.
Similarly, the pattern of a vector is the vector obtained from it by replacing each nonzero entry by a one.
The \emph{support} of a vector is the set of coordinates where it is nonzero.
The \emph{weight} of a vector is the size of its support, that is, the number of ones in its pattern.
The \emph{pattern} of a quantum Latin square is the matrix that is obtained from it by replacing each entry by its pattern.
For example:

\begin{table}[H]
    \centering
\begin{tabular}{|c|c|c|c|}
    \hline
    \(\ket{1}\) & \(\ket{2}\) & \(\ket{3}\) & \(\ket{4}\)\\
    \hline
    \(\ket{2}\) & \(\ket{1}\) & \(\ket{4}\) & \(\ket{3}\)\\
    \hline
    \(\ket{3}\) & \(\ket{4}\) & \(\frac1{\sqrt{2}}\left(\ket{1}+\ket{2}\right)\) & \(\frac1{\sqrt{2}}\left(\ket{1}-\ket{2}\right)\)\\
    \hline
    \(\ket{4}\) & \(\ket{3}\) & \(\frac1{\sqrt{2}}\left(\ket{1}-\ket{2}\right)\) & \(\frac1{\sqrt{2}}\left(\ket{1}+\ket{2}\right)\)\\
    \hline
\end{tabular}
\quad
\(\xrightarrow{\text{pattern}}\)
\quad
\begin{tabular}{|c|c|c|c|}
    \hline
    1000 & 0100 & 0010 & 0001\\
    \hline
    0100 & 1000 & 0001 & 0010\\
    \hline
    0010 & 0001 & 1100 & 1100\\
    \hline
    0001 & 0010 & 1100 & 1100\\
    \hline
\end{tabular}
\end{table}

Note that the pattern depends on the chosen orthonormal basis \(\left\{\ket{1},\dots,\ket{n}\right\}\).
We recall the following ``quantum sudoku rules'' that were introduced in \cite{36paper}. We refer to that paper for a more detailed explanation.

\begin{mdframed}
\begin{enumerate}
    \item[\#1] (Standard form.) We can apply unitary transformations and assume that, up to isotopy, the squares are in standard form where the first row in the pattern of each quantum Latin square is
    \[\begin{tabular}{|c|c|c|c|}
    \hline
    10\dots0 & 010\dots0 & \(\cdots\) & 0\dots01\\
    \hline
    \end{tabular}\]
    \item[\#2] (Unitary patterns.) Each row and each column of a quantum Latin square form a unitary pattern, by definition of quantum Latin squares. In particular, if there is an entry of weight one with support \(\{k\}\), then all other entries in the same row and column have a zero in position \(k\).
    \item[\#3] (Zero overlap.) From the second row onward, at a fixed cell \((i,j)\) and coordinate \(k\), two orthogonal squares cannot both have a pattern in that cell with a 1 in the \(k\)th coordinate. Indeed, since the quantum Latin squares are mutually orthogonal, \(\psi_{ij}\otimes\phi_{ij}\) and \(\psi_{1k}\otimes\phi_{1k}=\ket{k}\otimes\ket{k}\) are orthogonal, so \(\braket{k}{\psi_{ij}}\braket{k}{\phi_{ij}}=0\).
\end{enumerate}
\end{mdframed}

\subsection{Frobenius rings}

Let $R$ be a finite ring of order \(n\) with identity. We denote the set of its invertible elements by \(R^*\). An \emph{additive character} of $R$ is a group morphism \(\chi:(R,+)\to(\mathbb{T},\cdot)\) where \(\mathbb{T}=\{z\in\mathbb{C}:|z|=1\}\) is the complex unit circle.

A \emph{Frobenius ring} is a ring $R$ that admits a \emph{generating character}: an additive character $\chi$ such that every additive character of \(R\) is uniquely of the form \(x\mapsto\chi(ax)\) for some \(a\in R\).

Examples of Frobenius rings are the modular ring \(\ZZ/n\ZZ\) (with generating character \(x\mapsto\omega^x\), where \(\omega\) is a primitive \(n\)th root of unity) and the finite field \(\FF_q\) (with generating character \(x\mapsto\omega^{\tr(x)}\), where \(\omega\) is a primitive \(p\)th root of unity and \(q=p^h\), \(p\) prime), see also \cite{Frobenius}.

\begin{lemma}\label{lem:frobenius}
If \(\chi\) is a generating character, then \(\sum_{x\in R}\chi(ax)=0\) for all \(a\in R\setminus\{0\}\).
\end{lemma}
\begin{proof}
Since \(x\mapsto\chi(ax)\) is non-trivial (the trivial character is represented uniquely by \(a=0\)), there exists an element \(b\in R\) with \(\chi(ab)\neq1\). If we define \(S=\sum_{x\in R}\chi(ax)\), then
\[S=\sum_{x\in R}\chi(a(x+b))=\sum_{x\in R}\chi(ax)\chi(ab)=\chi(ab)S,\]
because $\chi$ is additive. Since \(\chi(ab)\neq1\), we conclude that $S=0$.
\end{proof}

\section{\texorpdfstring{A set of \(n-2\) \MOQLS{n} is classical}{A set of n-2 MOQLS(n) is classical}}\label{sec:n-2}

We make use of the following classical result.

\begin{theorem}[\cite{Shrikhande}]\label{thm:bruckcompletion}
    A set of \(n-3\) \MOLS{n} can be extended to a set of \(n-1\) \MOLS{n} if \(n\neq4\).
\end{theorem}

In particular, a set of \(n-2\) \MOLS{n} can be extended to a set of \(n-1\) \MOLS{n}. This also holds for \(n=4\), which is not so difficult to verify directly.

\begin{lemma}[{\cite[Lemma~13]{36paper}}]\label{lem:weight2}
    If \(\Psi\) is a quantum Latin square without entries of weight three or more, then there exists a classical quantum Latin square \(\Psi'\) such that whenever \(\Psi\) and \(\Phi\) are orthogonal, \(\Psi'\) and \(\Phi\) are orthogonal as well.
\end{lemma}
\begin{proof}
    Let \(n\) be the order of \(\Psi\). We may assume that \(n\geq4\), otherwise the statement is trivial.
    We show that, as long as \(\Psi\) has entries of weight two, we can replace \(\Psi\) by a quantum Latin square with strictly fewer entries of weight two and no entries of weight three or more, while still being orthogonal to \(\Phi\). We can repeat this operation until all entries have weight one.
    
    Choose an entry \(\psi_{ij}\) of weight two, and assume without loss of generality that it has pattern 110\dots0. Let \(U\) be the unique (up to phase factors) unitary transformation that converts \(\psi_{ij}\) into \(\ket{1}\) and fixes \(\ket{3},\dots,\ket{n}\).
    Let \(\Psi'\) be the quantum Latin square obtained from \(\Psi\) by applying \(U\) to every entry with pattern 110\dots0, while keeping the other entries the same.
    First, the operation never increases the weight of any entry, and decreases the weight of \(\psi_{ij}\). Second, the operation preserves the orthogonality with other quantum Latin squares because whenever two entries in \(\Psi\) are orthogonal, the corresponding entries in \(\Psi'\) are orthogonal as well. The only case where this could go wrong is when one of the two orthogonal entries has pattern 110\dots0, but then the other entry has either the same support \(\{1,2\}\) or a support disjoint from \(\{1,2\}\). 
\end{proof}

We are now ready to prove Theorem~\ref{thm:n-2}.

\begin{proof}[Proof of Theorem~\ref{thm:n-2}]
    Suppose, for a contradiction, that \(\{\Psi_1,\dots,\Psi_{n-2}\}\) is a set of \(n-2\) MOQLS of order \(n\) where at least one of them, say \(\Psi_1\), is not classical. By rule \#1, we may assume them to be in standard form. At an off-first-row cell \((i,j)\), rule \#2 excludes coordinate \(j\) from every support, while rule \#3 makes the supports from the \(n-2\) squares pairwise disjoint. Therefore, all entries have weight at most \((n-1)-(n-3)=2\). Since \(\Psi_1\) is assumed to be non-classical, it has an entry of weight two. The first row has only entries of weight one, so suppose without loss of generality that the entry \((\Psi_1)_{21}\) has pattern 0110\dots0. In the same row, the other \(n-1\) patterns cannot be all of the form *00*\dots*, because they are orthogonal and hence span a space of dimension \(n-1\). So there exists a column \(j\) for which \((\Psi_1)_{2j}\) has a pattern of the form 01*\dots* or 0*1*\dots*. Since \((\Psi_1)_{21}\) and \((\Psi_1)_{2j}\) are orthogonal and the weight of \((\Psi_1)_{2j}\) is at most two, the pattern of \((\Psi_1)_{2j}\) is equal to 0110\dots0 as well. Similarly, there is another entry \((\Psi_1)_{i1}\) with pattern 0110\dots0 in the first column.
    \[
    \Psi_1=
    \begin{tabular}{|c|c|c|c|}
    \hline
    1000\dots0 & & &\\
    \hline
    0110\dots0 & \dots & 0110\dots0 &\dots\\
    \hline
    \vdots & & &\\
    \hline
    0110\dots0 & & &\\
    \hline
    \vdots & & &\\
    \hline
    \end{tabular}
    \quad\mapsto\quad
    \Psi_1'=
    \begin{tabular}{|c|c|c|c|}
    \hline
    1000\dots0 & & &\\
    \hline
    0100\dots0 & \dots & 0010\dots0 &\dots\\
    \hline
    \vdots & & &\\
    \hline
    0010\dots0 & & &\\
    \hline
    \vdots & & &\\
    \hline
    \end{tabular}
    \]
    Now apply Lemma~\ref{lem:weight2} successively to all \(n-2\) squares to obtain a set \(\{\Psi_1',\Psi_2',\dots,\Psi_{n-2}'\}\) of \(n-2\) classical \MOQLS{n}. Note that every resulting support is contained in the corresponding original support. Choose the first unitary used in the proof of Lemma~\ref{lem:weight2} for \(\Psi_1\) so that \((\Psi_1')_{21}=\ket{2}\). The entries \((\Psi_1)_{2j}\) and \((\Psi_1)_{i1}\) have the same support of weight two and are orthogonal to \((\Psi_1)_{21}\), so they are mapped to \(\ket{3}\) up to a phase.

    For every \(r\geq2\), rule \#3 shows that the original entries \((\Psi_r)_{2j}\) and \((\Psi_r)_{i1}\) have support disjoint from \(\{2,3\}\). Support preservation in Lemma~\ref{lem:weight2} implies that the final classical entries at those cells in \(\Psi_r'\) are neither \(\ket{2}\) nor \(\ket{3}\).
    
    By Theorem~\ref{thm:bruckcompletion} (and the remark below it), there is a classical quantum Latin square \(\Phi\) of order \(n\) such that \(\{\Psi_1',\dots,\Psi_{n-2}',\Phi\}\) is a set of \(n-1\) classical \MOQLS{n} in standard form.
    At position \((2,j)\), the set of entries of \(\Psi_1',\dots,\Psi_{n-2}'\) are precisely the elements of \(\{\ket{1},\ket{3},\ket{4},\dots,\ket{n}\}\setminus\{\ket{j}\}\), up to phase factors, so \(\Phi_{2j}=\ket{2}\). Similarly, the set of entries of \(\Psi_1',\dots,\Psi_{n-2}'\) at position \((i,1)\) is \(\{\ket{3},\dots,\ket{n}\}\), up to phase factors, so \(\Phi_{i1}=\ket{2}\). Now both \(\left(\psi_1'\right)_{2j}\otimes\phi_{2j}\) and \(\left(\psi_1'\right)_{i1}\otimes\phi_{i1}\) are equal to \(\ket{3}\otimes\ket{2}\), contradicting the orthogonality between \(\Psi_1'\) and \(\Phi\).
\end{proof}

\section{Large non-classical sets of MOQLS}\label{sec:construction}

Instead of using the index set \(\{1,\dots,n\}\) for the elements of the fixed orthonormal basis and for the rows and columns of a quantum Latin square of order \(n\), we can also use any finite Frobenius ring \(R\) of order \(n\). More precisely, we can denote the standard basis of \(\CC^n\) by \(\{\ket{i}\colon\,i\in R\}\) and the entries of a quantum Latin square \(\Psi\in(\CC^n)^{n\times n}\) by \(\psi_{ij}\) with \(i,j\in R\).

Let \(R\) be a finite Frobenius ring of order \(n\) and let \(\chi:R\to\mathbb{T}\) be a generating character. Recall that \(\mathbb{T}=\{z\in\mathbb{C}:|z|=1\}\) is the complex unit circle. 

\begin{definition}\label{def:construction}
    Given a permutation \(f:R\to R\), we define the quantum Latin square \(\Psi^{(f)}=(\psi_{ij}^{(f)})\) as
    \[\psi_{ij}^{(f)}=\frac1{\sqrt{n}}\sum_{x\in R}\chi(if(x)+jx)\ket{x}.\]
\end{definition}

We first prove that \(\Psi^{(f)}\) is indeed a quantum Latin square:

\begin{lemma}\label{lem:construction}
If \(f:R\to R\) is a permutation, then \(\Psi^{(f)}\) is a quantum Latin square.
\end{lemma}
\begin{proof}
For all $j\neq j'$ we have
\[\braket{\psi_{ij}^{(f)}}{\psi_{ij'}^{(f)}}=\frac1n\sum_{x\in R}\chi((j'-j)x)=0\]
by Lemma~\ref{lem:frobenius}.
Similarly, for $i\neq i'$ we have
\[
\braket{\psi_{ij}^{(f)}}{\psi_{i'j}^{(f)}}=\frac1n\sum_{x\in R}\chi((i'-i)f(x))=0
\]
by Lemma~\ref{lem:frobenius}, since \(f\) is a permutation of \(R\).
\end{proof}

\begin{lemma}\label{lem:classical}
    Let \(f:R\to R\) be a permutation. The quantum Latin square \(\Psi^{(f)}\) is classical if and only if \(f\) is the sum of an additive map and a constant.
\end{lemma}
\begin{proof}
    Replacing \(f\) by \(f-f(0)\) changes every entry \(\psi_{ij}^{(f)}\) only by the phase factor \(\chi(if(0))\), so assume that \(f(0)=0\).
    Since every row of a quantum Latin square of order \(n\) already contains \(n\) distinct entries, \(\Psi^{(f)}\) is classical if and only if its total cardinality (the number of distinct entries up to phase) is \(n\).

    First assume that \(\Psi^{(f)}\) is classical. The first row and column must have the same entries up to phase, so for every \(i\in R\), there exists a phase factor \(c\in\mathbb{T}\) and a column index \(j\in R\) such that
    \[\frac1{\sqrt{n}}\sum_{x\in R}\chi(if(x))\ket{x}=\frac{c}{\sqrt{n}}\sum_{x\in R}\chi(jx)\ket{x}.\]
    Comparing coefficients gives that, for every \(i\in R\), there are \(c\in\mathbb{T}\) and \(j\in R\) such that
    \[\chi(if(x))=c\chi(jx)\]
    for all \(x\in R\).
    Filling in \(x=0\), we get that \(c=1\) because \(f(0)=0\) and \(\chi(0)=1\). 
    Since \(x\mapsto \chi(jx)\) is an additive character, the map \(x\mapsto\chi(if(x))\) is additive character. That is, \(\chi(if(x+y))=\chi(if(x))\cdot\chi(if(y))\) for all \(i,x,y\in R\). By additivity of \(\chi\), we get \(\chi(i(f(x+y)-f(x)-f(y)))=1\) for all \(i,x,y\in R\). But \(\chi\) is a generating character, so \(f(x+y)=f(x)+f(y)\) for all \(x,y\in R\). In other words, \(f\) is additive.
    
    Conversely, if \(f\) is additive, then for every \(i\in R\), the map \(x\mapsto\chi(if(x))\) is an additive character, so by the defining property of a generating character, there exists a unique element \(a\in R\) such that \(\chi(if(x))=\chi(ax)\) for \(x\in R\). Thus, \(\psi_{ij}^{(f)}=\frac1{\sqrt{n}}\sum_{x\in R}\chi((a+j)x)\ket{x}\) is equal to the first-row entry on position \((0,a+j)\). We conclude that \(\Psi^{(f)}\) has cardinality \(n\) and is therefore classical.
\end{proof}

If \(f(x)=mx\) with \(m\in R^*\), then \(\Psi^{(f)}\) is isotopic to the classical quantum Latin square \(\Phi^{(m)}=(\phi_{ij}^{(m)})\) with \[\phi_{ij}^{(m)}=\ket{mi+j}\]
via the Fourier transform \(F:\CC^n\to\CC^n:\ket{x}\mapsto\frac1{\sqrt{n}}\sum_{y\in R}\chi(xy)\ket{y}\).

\begin{lemma}\label{lem:orthogonal}
Let \(m\in R^*\). If both \(f(x)\) and \(x\mapsto f(x)-mx\) are permutations of \(R\), then \(\Psi^{(f)}\) and \(\Phi^{(m)}\) are orthogonal quantum Latin squares.
\end{lemma}
\begin{proof}
    By Lemma~\ref{lem:construction}, \(\Psi^{(f)}\) is a quantum Latin square, and \(\Phi^{(m)}\) is a quantum Latin square because \(m\) is invertible, so it only remains to prove that \(\psi_{ij}^{(f)}\otimes\phi_{ij}^{(m)}\) is orthogonal to \(\psi_{k\ell}^{(f)}\otimes\phi_{k\ell}^{(m)}\) for all \((i,j)\neq(k,\ell)\). That is,
    \[\frac1{\sqrt{n}}\sum_{x\in R}\chi(if(x)+jx)\ket{x}\otimes\ket{mi+j}\]
    is orthogonal to
    \[\frac1{\sqrt{n}}\sum_{y\in R}\chi(kf(y)+\ell y)\ket{y}\otimes\ket{mk+\ell}\]
    for all \((i,j)\neq(k,\ell)\). The inner product is zero immediately if \(mi+j\neq km+\ell\). When \(mi+j=km+\ell\), the inner product can be nonzero only when \(x=y\). Thus, the inner product is
    \[\frac1n\sum_{x\in R}\chi((j-\ell)x+(i-k)f(x))=\frac1n\sum_{x\in R}\chi(((i-k)(f(x)-mx))).\]
    Since \(x\mapsto f(x)-mx\) is a permutation of \(R\), this sum is zero by Lemma~\ref{lem:frobenius}, unless \(i=k\). But if \(i=k\), then \(j=\ell\), which is not the case.
\end{proof}

\begin{proof}[Proof of Theorem~\ref{thm:construction}]
    Let \(m\in M\). Replacing \(f\) by \(x\mapsto f(x)-mx\) and \(M\) by \(M-m\) if needed, we may assume that \(f\) is a permutation and \(0\in M\). By Lemma~\ref{lem:classical}, \(\Psi^{(f)}\) is non-classical. For every \(m\in M\setminus\{0\}\), we have that \(m\in R^*\) and that \(x\mapsto f(x)-mx\) is a permutation, so we can include the classical square \(\Phi^{(m)}\). Lemma~\ref{lem:orthogonal} shows that each of them is orthogonal to \(\Psi^{(f)}\). If \(m\neq m'\), then the inner product of \(\ket{mi+j}\otimes\ket{m'i+j}\) and \(\ket{mk+\ell} \otimes \ket{m'k+\ell}\) is zero unless $mi+j=mk+\ell$ and $m'i+j=m'k+\ell$, which implies $(m-m')(i-k)=0$, so \(i=k\) and then $j=\ell$. Hence the classical squares are mutually orthogonal. The resulting set contains \(|M|\) quantum Latin squares, exactly one of which is non-classical.
\end{proof}

\subsection{Functions over a finite field that determine few directions}

In this subsection, we consider the case when \(R=\FF_q\) is a finite field, so \(n=q=p^h\) where \(p\) is prime. See \cite{Frobenius} for a proof that a finite field is indeed a Frobenius ring.

In view of Theorem~\ref{thm:construction}, we want to find permutations \(f:\FF_q\to\FF_q\) that are not additive (\(\FF_p\)-linear) up to a constant, and that determine few directions, where we say that \[\left\{\frac{f(x)-f(y)}{x-y}\colon\,x,y\in\FF_q, x\neq y\right\}\] is the set of directions determined by \(f\). Indeed, observe that, for \(x\neq y\), 
\[f(x)-mx=f(y)-my\quad\Longleftrightarrow\quad m=\frac{f(x)-f(y)}{x-y},\]
so \(f\) does not determine the direction \(m\in\FF_q\) if and only if \(x\mapsto f(x)-mx\) is a permutation of \(\FF_q\).

\begin{example}[in \cite{BBS} this example is attributed to Megyesi]\label{ex:f2}
    Suppose that \(H\) is a proper multiplicative subgroup of \((\FF_q\setminus\{0\},\cdot)\) with \(|H|\geq1\). Define \[f:\FF_q\to\FF_q:x\mapsto\begin{cases}x&\text{if }x\in H,\\0&\text{if }x\notin H.\end{cases}\]
    Thus,
    \[\frac{f(x)-f(y)}{x-y} \in \{0,1\}\]
    if both \(x\) and \(y\) are elements of \(H\) or both are not elements of \(H\). In the case \(x\in H\) and \(y\notin H\), we have
    \[\frac{f(x)-f(y)}{x-y}=\frac{x}{x-y}=\frac{1}{1-z}\]
    where \(z=y/x\) (note that \(x\neq0\)) is not contained in \(H\). If \(x\notin H\) and \(y\in H\), we have
    \[\frac{f(x)-f(y)}{x-y}=\frac{-y}{x-y}=\frac{1}{1-z}\]
    where \(z=x/y\) (note that \(y\neq0\)) is not contained in \(H\). Therefore, the function $f$ determines exactly \(2+(q-1-|H|)=q+1-|H|\) directions. Thus, there are \(|H|-1\) undetermined directions.
\end{example}

To maximize the number of undetermined directions, we can choose \(|H|\) to be the largest proper divisor of \(q-1\), assuming \(q-1\) is not prime, because \((\FF_q\setminus\{0\},\cdot)\) is a cyclic group. In particular, for odd \(q\), there are $(q-3)/2$ undetermined directions.

Rédei \cite{Redei} proved that when \(q\) is an odd prime, any non-affine function determines at least $(q+3)/2$ directions, so this is the best we can do. Moreover, Lov\'asz and Schrijver \cite{LS1981} proved that, if \(q\) is an odd prime, any optimal non-affine function is affinely equivalent to the function \(f:\FF_q\to\FF_q:x\mapsto x^{(q+1)/2}\).

\begin{proof}[Proof of Corollary~\ref{cor:construction}]
    We use Theorem~\ref{thm:construction} together with Example~\ref{ex:f2}, in the optimal case when \(|H|\) is the largest proper divisor of \(q-1\).
    When applying Theorem~\ref{thm:construction}, we let \(M\) be the set of undetermined directions. Note that the second condition on \(M\) is trivially fulfilled.

    To complete the proof, we verify that the function \(f\) from Example~\ref{ex:f2} is non-additive (that is, not \(\FF_p\)-linear). 
    Since \(H\) is a proper subgroup of \((\FF_q\setminus\{0\},\cdot)\), there exists \(x\in\left(\FF_q\setminus\{0\}\right)\setminus H\). Additivity would imply \(f(x+1)=f(1)=1\), but also \(f(x+1)\in\{0,x+1\}\) by definition of \(f\), a contradiction.
\end{proof}

\section{Conclusion and future directions}

We constructed new lower and upper bounds on the largest number \(t\) for which there exist non-classical sets of \(t\) \MOQLS{n}. Specific values of these new bounds for small \(n\) can be found in Table~\ref{tab:newbounds}.

We pose two open problems for future research:

\begin{problem}
    Is a set of \(n-3\) \MOQLS{n}, \(n\neq4\), necessarily classical?
\end{problem}

\begin{problem}
    Are there numbers \(t\) and \(n\) for which there exist \(t\) \MOQLS{n} but no \(t\) \MOLS{n}?
\end{problem}

The smallest open case for the second problem is \(n=10\). It is not known that \(7\) \MOLS{10} do not exist (combining Theorem~\ref{thm:bruckcompletion} with the non-existence of a projective plane of order 10, see \cite{BruckRyser}), but it is still an open problem whether \(7\) \MOQLS{10} exist. It would already be interesting to find \(3\) \MOQLS{10}, because the existence of \(3\) \MOLS{10} is not known.

\paragraph{Acknowledgements.}
The authors acknowledge the support of the Spanish Ministry of Science, Innovation and Universities grant PID2023-147202NB-I00.
Robin Simoens is supported by the Research Foundation Flanders (FWO) through the grant 11PG724N.
We thank Eline Herdies and Leo Storme for helpful discussions towards proving Theorem~\ref{thm:n-2}. We used GPT-5.5 Pro by OpenAI to proofread the manuscript.

\bibliographystyle{abbrv}
\bibliography{ref}

\vfill
\noindent\textsc{Simeon Ball}\\
\textsc{\small Department of Mathematics}\\[-1mm]
\textsc{\small Universitat Politècnica de Catalunya}\\[-1mm]
\textsc{\small C. Pau Gargallo 14, 08028 Barcelona, Spain}\\
{\it E-mail address:} {\href{mailto:simeon.michael.ball@upc.edu}{\url{simeon.michael.ball@upc.edu}}}\\

\noindent\textsc{Robin Simoens}\\
\textsc{\small Department of Mathematics: Analysis, Logic and Discrete Mathematics}\\[-1mm]
\textsc{\small Ghent University}\\[-1mm]
\textsc{\small Krijgslaan 297, 9000 Gent, Belgium}\\
\textsc{\small Department of Mathematics}\\[-1mm]
\textsc{\small Universitat Politècnica de Catalunya}\\[-1mm]
\textsc{\small C. Pau Gargallo 14, 08028 Barcelona, Spain}\\
{\it E-mail address:} {\href{mailto:Robin.Simoens@UGent.be}{\url{Robin.Simoens@UGent.be}}}\\

\end{document}